\documentclass{amsart}

\makeatletter
\@namedef{subjclassname@2010}{%
  \textup{2010} Mathematics Subject Classification}
\makeatother

\usepackage{amssymb}
\usepackage{amsfonts}
\usepackage{amsmath}
\usepackage{amsthm}
\usepackage{mathrsfs}
\usepackage{dsfont}
\usepackage{theoremref}

\newtheorem{thm}{Theorem}[section]
\newtheorem{lem}[thm]{Lemma}
\newtheorem{cor}[thm]{Corollary}
\newtheorem{prp}[thm]{Proposition}

\theoremstyle{definition}
\newtheorem{dfn}[thm]{Definition}

\usepackage{a4wide}

\author{Tristan Bice}
\address{
Mathematical Logic Group\\
Kobe University\\
Japan}
\email{tristan\_bice@kurt.scitec.kobe-u.ac.jp}
\thanks{This research has been supported by the Japanese Ministry of Education, Culture, Sports, Science and Technology (Mombukagakusho)}


\begin{document}

\title{The Order on Projections in C$^*$-Algebras of Real Rank Zero}

\begin{abstract}
We prove a number of fundamental facts about the canonical order on projections in C$^*$-algebras of real rank zero.  Specifically, we show that this order is separative and that arbitrary countable collections have equivalent (in terms of their lower bounds) decreasing sequences.  Under the further assumption that the order is countably downwards closed, we show how to characterize greatest lower bounds of finite collections of projections, and their existence, using the norm and spectrum of simple product expressions of the projections in question.  We also characterize the points at which the canonical homomorphism to the Calkin algebra preserves least upper bounds of countable collections of projections, namely that this occurs precisely when the span of the corresponding subspaces is closed.
\end{abstract}

\maketitle

\section{Introduction}

Real rank zero C$^*$-algebras have been studied for over two decades, since they were first defined in \cite{p}.  Furthermore, specific examples of such C$^*$-algebras have been an object of study long before that, like the Calkin algebra, for example, which was first defined in \cite{o} another 50 years earlier.  Consequently, it is somewhat surprising that the basic questions we answer in this paper have not been dealt with before.  More recently it has come to light that certain famous problems, like the Kadison-Singer conjecture for example, have an equivalent formulation in terms of the order on projections (see \cite{q}).  Given that even a very basic understanding of this order structure has been lacking, it is perhaps no wonder that a problem like the Kadison-Singer conjecture has remained open for over half a century.

For example, one can ask the following simple questions.  When do two projections have a non-zero lower bound?  When do they have a g.l.b. (greatest lower bound)?  How can this g.l.b. be characterized?  What about l.u.b.s (least upper bounds)?  What about arbitrary finite and countable collections?  When they do not have a g.l.b., how pathological do things get?  Can we even have finite gaps, for example? 

Before we detail our approach to these questions, let us step back a little and look at what is already well known. Take a Hilbert space $H$ (over the scalar field $\mathbb{F}=\mathbb{R}$ or $\mathbb{C}$) and let $\overline{\mathcal{V}}(H)$ denote the collection of closed subspaces of $H$.  It is common knowledge that $\overline{\mathcal{V}}(H)$ is a lattice w.r.t. $\subseteq$, specifically $U\cap V$ and $\overline{U+V}$ will be the g.l.b. and l.u.b. of any $U,V\in\overline{\mathcal{V}}(H)$.  Also, $\overline{\mathcal{V}}(H)$ is separative w.r.t. $\subseteq$, for if $U\nsubseteq V$ then $U\cap(U\cap V)^\perp$ is a non-empty subspace of $U$ with zero intersection with $V$.  Thanks to the one-to-one correspondence between projections on $H$ and closed subspaces of $H$, which is also an order isomorphism (i.e. $\mathcal{R}(P)\subseteq\mathcal{R}(Q)\Leftrightarrow PQ=P$), the same applies to the projections $\mathcal{P}(\mathcal{B}(H))$ on $H$ and, more generally, to the projections $\mathcal{P}(A)$ in $A$, for any subalgebra $A$ of $\mathcal{B}(H)$ closed in the weak operator topology, i.e. a von Neumann algebra.

Classical works like \cite{j} and \cite{k} examine the existence of g.l.b.s and l.u.b.s with respect to the collection of \emph{all} self-adjoint elements $\mathcal{S}(A)$ of $A$ (with the order defined by $S\leq T\Leftrightarrow T-S$ is positive).  The situation is substantially different in this case, as pairs of projections in a von Neumann algebra $A$ can only have a g.l.b. with respect to $\mathcal{S}(A)$ if they commute, by \cite{j} Corollary 4 (we extend this to arbitrary unital C$^*$-algebras in \thref{PQncom}).  Even commutativity is not enough to guarantee that they have a g.l.b., as pairs of projections on $H$ have a g.l.b. with respect to $\mathcal{S}(\mathcal{B}(H))$ only if they are comparable, by \cite{j} Lemma 7.

This paper focuses on going in the other direction, expanding the range of C$^*$-algebras under consideration rather than the partial order.  Specifically, we examine the existence of g.l.b.s and l.u.b.s, still with respect to $\mathcal{P}(A)$, but in the more general case when $A$ is only assumed to have real rank zero.  Our primary motivation is the study of the projections in the Calkin algebra of $H$, the quotient $\mathcal{C}(H)=\mathcal{B}(H)/\mathcal{K}(H)$ where $\mathcal{K}(H)$ is the collection of compact operators.  This algebra is well known to have real rank zero (see \cite{l}), even though it is not a von Neumann algebra.  Indeed, it is shown in \cite{b} Proposition 5.26 that $\mathcal{P}(\mathcal{C}(H))$ is not even a lattice.  The order structure of $\mathcal{P}(\mathcal{C}(H))$ has also been investigated in \cite{n} and \cite{e} with a particular focus on linearly ordered (and hence commutative) gaps.  Indeed, all such work until now has dealt only with commutative subsets of $\mathcal{P}(\mathcal{C}(H))$, in particular on subsets of projections coming from the canonical embedding (with respect to some basis $(e_n)$ of $H$) of subsets of $\mathscr{P}(\omega)/\mathrm{Fin}$ into $\mathcal{C}(H)$.  No doubt this is because finite commutative subsets of $\mathcal{P}(A)$, for arbitrary C$^*$-algebra $A$, do necessarily have a g.l.b. given simply by their product.  Despite the fact this may fail for non-commutative subsets, we show in this paper that these subsets are not as intractable as they might at first appear, at least when they are countable and $A$ has real rank zero.

The reason for this is that real rank zero C$^*$-algebras are somewhat close to being von Neumann algebras in the sense that, while spectral projections of self-adjoint elements do not always exist in the algebra itself, arbitrarily close approximations do.  This allows some properties of $\mathcal{P}(A)$ to still be proved in this more general context, albeit with more effort using these spectral projection approximations.  For example, in \thref{sep} we show that $\mathcal{P}(A)$ must still be separative, while in \thref{arbcount} we show that the lower bounds of an arbitrary countable subset of $\mathcal{P}(A)$ are precisely the lower bounds of some (possibly non-strictly) decreasing sequence in $\mathcal{P}(A)$.  This allows some consequences of being a lattice to still be proved, despite the fact $\mathcal{P}(A)$ may not be a lattice.  For example we show in \thref{nogap} that $\mathcal{P}(A)$ having no atoms or $(\omega,\omega)$-gaps still implies it has no non-trivial countable gaps.  In \S\ref{ub} we investigate l.u.b.s, showing in (\ref{glblub}) that, when dealing with a C$^*$-algebra $A$ of real rank zero such that $\mathcal{P}(A)$ is $\sigma$-closed, a pair of projections has an l.u.b. if and only if it has a g.l.b..  We also point out some connections to sums of subspaces, showing in \thref{findim} that a countable sum of subspaces in $H$ is closed if and only if the canonical $\pi:\mathcal{B}(H)\rightarrow\mathcal{C}(H)$ preserves the l.u.b. of the corresponding projections.  Finally, we prove one result, \thref{PQncom}, about the existence of g.l.b.s of pairs of projections with respect to the collection of all self-adjoint elements, extending \cite{j} Corollary 4 from von Neumann algebras to arbitrary unital C$^*$-algebras.

\section{Spectral Families}

We first discuss the spectral families that will be used throughout this paper.

\begin{dfn}\thlabel{sf}
If $S\in\mathcal{S}(\mathcal{B}(H))$, for some Hilbert space $H$, then $E_S:\mathbb{R}\rightarrow\mathcal{P}(\mathcal{B}(H))$ is the \emph{spectral family} of $S$ if $\langle Sv,v\rangle\leq t$, for all unit $v\in\mathcal{R}(E_S(t))$, and $t<\langle Sv,v\rangle$, for all unit $v\in\mathcal{R}(E^\perp_S(t))$.
\end{dfn}

Our reference for spectral families is \cite{a}, where the existence and uniqueness of the spectral family $E_S$, for any $S\in\mathcal{S}(\mathcal{B}(H))$, is proved in \cite{a} Theorem 7.17.  Note that the case $\mathbb{F}=\mathbb{R}$ is proved here too, even though the spectral theorem for arbitrary normal operators requires $\mathbb{F}=\mathbb{C}$.  We also denote by $E_S(t-)$ the limit, approaching $t$ from below, of $E_S$ in the weak (or strong) operator topology.  Equivalently, $E_S(t-)$ can be given the same definition as in \thref{sf}, just with the $\leq$ and $<$ exchanged.  Also note that spectral families are usually defined in a different way, together with spectral measures and integrals.  However, we have chosen to define spectral families here simply in terms of the above elementary inequalities as these inequalities are the only spectral family properties used throughout this paper.  For example, we can use them to show that
\begin{equation}\label{PQP}
E^\perp_{PQP}(1-)=P\wedge Q,
\end{equation}
for projections $P,Q\in\mathcal{B}(H)$.  To see this, simply note that, for unit $v\in\mathcal{R}(P\wedge Q)=\mathcal{R}(P)\cap\mathcal{R}(Q)$, $\langle PQPv,v\rangle=\langle v,v\rangle=1$, while for unit $v\in H\backslash\mathcal{R}(P\wedge Q)$ we either have $v\notin\mathcal{R}(P)$, and hence $\langle PQPv,v\rangle\leq||PQPv||\leq||Pv||<1$, or $v\in\mathcal{R}(P)\backslash\mathcal{R}(Q)$, so $\langle PQPv,v\rangle=\langle Qv,v\rangle\leq||Qv||<1$.%, from which (\ref{PQP}) now follows.

One important point about spectral familes of elements of a C$^*$-algebra $A$ is that they depend on the particular Hilbert space on which $A$ is represented, even though some inequalities relating to them do not.  Specifically, given a C$^*$-algebra $A\subseteq\mathcal{B}(H)$, a faithful representation $\pi$ of $A$ on another Hilbert space $H'$, $S\in\mathcal{B}(H)$ and $t\in\mathbb{R}$, it may well be that $E_S(t)$ does not lie in $A$ and, even if it did, there would still be no guarantee that $\pi(E_S(t))=E_{\pi(S)}(t)$.  However, if it does lie in $A$, we would necessarily have $\pi(E_S(t))\leq E_{\pi(S)}(t)$.  More generally, the statement $\pi(P)=E_{\pi(S)}(t)$, for $P,S\in A$, depends on the (faithful) representation $\pi$, while the statement $\pi(P)\leq E_{\pi(S)}(t)$ does not, as we now show.

\begin{prp}\thlabel{pispec}
Let $H$ and $H'$ be Hilbert spaces, let $A\subseteq\mathcal{B}(H)$ be a C$^*$-algebra and let $\pi:A\rightarrow\mathcal{B}(H')$ be a homomorphism\footnote{Note that when $H$ and $H'$ are Hilbert spaces, $A\subseteq\mathcal{B}(H)$ and we say $\pi:A\rightarrow\mathcal{B}(H')$ is a homomorphism, we mean that $\pi$ preserves the algebraic operations (addition, multiplication and the adjoint operation) and, if the identity operator $1_{\mathcal{B}(H)}$ is in $A$, $\pi(1_{\mathcal{B}(H)})=1_{\mathcal{B}(H')}$.  If $1_{\mathcal{B}(H)}\notin A$ then we can always extend $\pi:A\rightarrow\mathcal{B}(H')$ to a homomorphism $\pi':A+\mathbb{F}1_{\mathcal{B}(H)}\rightarrow\mathcal{B}(H')$ by setting $\pi'(a+\lambda1_{\mathcal{B}(H)})=\pi(a)+\lambda1_{\mathcal{B}(H')}$, for all $a\in A$ and $\lambda\in\mathbb{F}$.
}.  For all $P\in\mathcal{P}(A)$, $S\in\mathcal{S}(A)$ and $t\in\mathbb{R}$, $P\leq E_S(t)$ implies $\pi(P)\leq E_{\pi(S)}(t)$ and $E_S(t-)\leq P$ implies $E_{\pi(S)}(t-)\leq\pi(P)$.
\end{prp}

\begin{proof} Let $T=(S+(||S||+|t|)1)/(t+||S||+|t|)$, so $P\leq E_S(t)=E_T(1)$ (this equality follows easily from the uniqueness of the spectral family in \thref{sf}).  Note that $T$ is positive and, as $\mathcal{R}(E_T(1))$ is $T$-invariant (see \cite{a} p196 Proposition (1)), it follows from \thref{sf} that $||T^nE_T(1)||\leq1$, for all $n\in\omega$.  As $||\pi||=1$, we also have $||\pi(T^nP)||\leq1$.  But $\pi(T)$ is also positive so by \cite{h}, $E_{\pi(T)}(1)$ is in fact the largest projection $p$ such that $||\pi(T)^np||\leq1$, for all $n\in\omega$.  Hence $\pi(P)\leq E_{\pi(T)}(1)=E_{\pi(S)}(t)$.  Thus the first implication is proved, and simply note that the second implication follows from the first and the fact that $E_S(t-)\leq P$ is equivalent to $P^\perp(=1-P)\leq E_{-S}(-t)$. \end{proof}

Alternatively, in the proof of \thref{pispec} one can use the continuous functional calculus which, unlike spectral families, is well defined in the abstract C$^*$-algebra context.  Specifically one can note that $E_S(t)=\inf_n(f_n(S))$, where $f_n$ is a sequence of non-negative continuous functions on $\mathbb{R}$ decreasing pointwise to the characteristic function of $(-\infty,t]$.  Then we see that $P\leq E_S(t)$ is equivalent to the statement that $P\leq f_n(S)$, for all $n\in\omega$, giving $\pi(P)\leq\pi(f_n(S))=f_n(\pi(S))$, for all $n\in\omega$, which is in turn equivalent to $\pi(P)\leq E_{\pi(S)}(t)$.

We will also mention the spectrum, $\sigma(T)$, of $T\in\mathcal{B}(H)$, as given below.
\begin{equation}\label{spec}
\sigma(T)=\{\lambda\in\mathbb{F}:T-\lambda 1\textrm{ is not invertible in }\mathcal{B}(H)\}.
\end{equation}
The importance of the spectrum in this paper comes from its relation to spectral families of $S\in\mathcal{S}(\mathcal{B}(H))$ given below (see \cite{a} Theorem 7.22 (iv)).
\begin{equation}\label{specE}
\sigma(S)=\{t\in\mathbb{R}:\forall\epsilon>0(E_S(t+\epsilon)-E_S(t-\epsilon)\neq0)\}.
\end{equation}

It is well known that if $T\in A$ for some C$^*$-algebra $A\subseteq\mathcal{B}(H)$ containing the identity operator $1$ then $\mathcal{B}(H)$ may be replaced with $A$ in (\ref{spec}) above (see \cite{m} p3).  It follows that if $\pi:A\rightarrow B$ is a homomorphism from a C$^*$-algebra $A$ to a C$^*$-algebra $B$ then $\sigma(\pi(T))\subseteq\sigma(T)$, for all $T\in A$.  In the case $A=\mathcal{B}(H)$ and $B=\mathcal{C}(H)$, $\sigma(\pi(T))$ is known to be precisely the limit points of $\sigma(T)$ together with the eigenvalues of $T$ of infinite multiplicity, so long as $\mathbb{F}=\mathbb{C}$ or $T\in\mathcal{S}(A)$.

We shall also use the following elementary results.  For $S\in\mathcal{S}(A)$, $||S||=\max|\sigma(S)|$, which can be derived from \thref{sf}, (\ref{specE}) and the fact that $||S||=\sup_{||v||=1}|\langle Sv,v\rangle|$ (see \cite{a} Theorem 4.4 (b)) so, for $P,Q,\in\mathcal{P}(A)$,
\begin{equation}\label{||PQ||}
||PQ||^2=||PQ(PQ)^*||=||PQQP||=||PQP||=\max(\sigma(PQP)).
\end{equation}
Note that $0\in\sigma(TP)$, for any $T\in A$ and $P\in\mathcal{P}(A)\backslash\{1\}$.  Also, for any $S,T\in A$, we have $\sigma(ST)\backslash\{0\}=\sigma(TS)\backslash\{0\}$ (see \cite{m} Exercise 1.7.A.) which means, for $P,Q\in\mathcal{P}(A)$,
\begin{equation}\label{sigmaPQP}
\sigma(PQP)=\sigma(PPQ)=\sigma(PQ)=\sigma(QP)=\sigma(QQP)=\sigma(QPQ).
\end{equation}
For $S\in\mathcal{S}(A)$ and $P\in\mathcal{P}(A)$ with $SP=S=PS$, $E_{P-S}(t)=E^\perp_S(1-t-)+P^\perp$, for $t\in(0,1)$.  By (\ref{specE}), this implies that, for $Q\in\mathcal{P}(A)$,
\begin{equation}\label{sigmaPQperpP}
\sigma(PQ^\perp P)\cap(0,1)=\sigma(P-PQP)\cap(0,1)=1-\sigma(PQP)\cap(0,1).%=\{t\in(0,1):1-t\in\sigma(PQP)\}.
\end{equation}

\section{Quotients of C$^*$-algebras of Real Rank Zero}

As stated in the title, we will primarily be interested in C$^*$-algebras of the following type.

\begin{dfn}
A C$^*$-algebra $A$ has \emph{real rank zero} if $\mathcal{S}^{-1}(A)$, the collection of self-adjoint invertible elements of $A$ (with a unit adjoined, if $A$ has none), is dense in $\mathcal{S}(A)$.
\end{dfn}

\thref{pirrz} below follows from this definition and some elementary C$^*$-algebra theory.

\begin{prp}\thlabel{pirrz}
If a C$^*$-algebra $A$ has real rank zero and $\pi$ is a homomorphism onto another C$^*$-algebra $B$ then $B$ also has real rank zero.
\end{prp}

From now on, rather than applying the definition of real rank zero directly, we will make use of the following important consequence of having real rank zero.

\begin{thm}\thlabel{IP}
Let $H$ be a Hilbert space and let $A\subseteq\mathcal{B}(H)$ be a C$^*$-algebra of real rank zero.  For all $S\in\mathcal{S}(A)$ and $0<t<s\in\mathbb{R}$ there exists $P\in\mathcal{P}(A)$ such that $E^\perp_S(s)\leq P\leq E^\perp_S(t)$.
\end{thm}

For a proof of this result, see \cite{g}, at least for the specific case where $H$ is the Hilbert space coming from the universal representation of $A$.  For the general result simply note that, by \thref{pispec}, if the spectral families in \thref{IP} can be taken with respect to some faithful representation of $A$, like the universal representation, then they can actually be taken with respect to any representation, be it faithful or not.  Alternatively, one can use the fact that any representation $\pi$ of a C$^*$-algebra $A$ can be extended to a normal representation $\pi''$ of the entirety of its universal enveloping algebra $A^{**}$ onto $\pi[A]''$ (see \cite{r} Theorem 3.7.7).  Normality implies that spectral projections with respect to the universal representation push forward to the same spectral projections, i.e. normality gives us the middle equality in $\pi''(E_S(t))=\pi''(\inf_nf_n(S))=\inf_n\pi(f_n(S))=E_{\pi(S)}(t)$, for all $S\in\mathcal{S}(A)$ and $t\in\mathbb{R}$, where the $(f_n)$ are as in the remark after \thref{pispec}, from which the full generality of \thref{IP} again follows.

\begin{thm}\thlabel{Q<P}
If $\pi$ is a C$^*$-algebra homomorphism from $A$ onto $B$, $A$ has real rank zero, $q\in\mathcal{P}(B)$, $P\in\mathcal{P}(A)$ and $q\leq\pi(P)$ then there is $Q\in\mathcal{P}(A)$ with $\pi(Q)=q$ and $Q\leq P$.
\end{thm}

\begin{proof} Take $T\in A$ such that $\pi(T)=q$.  Setting $S=(T+T^*)/2\in\mathcal{S}(A)$, we see that still $\pi(S)=(q+q^*)/2=q$.  Take positive $\delta<1/2$ and $Q\in\mathcal{P}(A)$ with $E^\perp_{PSP}(1-\delta)\leq Q\leq E^\perp_{PSP}(\delta-)$.\footnote{In fact, by \cite{i}, $A$ does not have to have real rank zero for this, so long as $d(PSP\pm1/2,\mathcal{S}^{-1}(A))<1/2-\delta$, where $d$ here is the (infimum) distance function.}  As $\pi(S)=q\leq\pi(P)$, $\pi(PSP)=q$ so, by \thref{pispec}, \[q=E^\perp_q(1-\delta)\leq\pi(Q)\leq E^\perp_q(\delta-)=q.\]  Also $\mathcal{N}(P)\subseteq\mathcal{N}(PSP)\subseteq\mathcal{R}(E_{PSP}(\delta-))$ so $P^\perp\leq E_{PSP}(\delta-)$ and hence $Q\leq E^\perp_{PSP}(\delta-)\leq P$. \end{proof}

In particular, when $P=1$ this shows that every projection in $B$ can be pulled back to a projection in $A$ if $A$ has real rank zero, which was proved in \cite{c} for $A=\mathcal{B}(H)$ and $B=\mathcal{C}(H)$.

We can use this to show further that a projection sandwiched between the image of two other projections can be pulled back to a projection still sandwiched between them.

\begin{cor}\thlabel{S<Q<P}
If $\pi$ is a C$^*$-algebra homomorphism from $A$ onto $B$, $A$ has real rank zero, $q\in\mathcal{P}(B)$, $P,R\in\mathcal{P}(A)$, $R\leq P$ and $\pi(R)\leq q\leq\pi(P)$ then $\exists Q\in\mathcal{P}(A)$ with $\pi(Q)=q$ and $R\leq Q\leq P$.
\end{cor}

\begin{proof} Applying \thref{Q<P}, we have $S\in\mathcal{P}(A)$ with $S\leq P$ and $\pi(S)=q$.  Then apply \thref{Q<P} again to get $T\in\mathcal{P}(A)$ with $T\leq P-R$ and $\pi(T)=\pi(P-S)$.  Setting $Q=P-T$, we see that $\pi(Q)=\pi(S)=q$ and $R\leq Q\leq P$. \end{proof}

\begin{dfn}
Take a partial order $\mathbb{P}$ with minimum element $0$ and $\mathcal{P},\mathcal{Q}\subseteq\mathbb{P}$.  If $p\in\mathbb{P}\backslash\{0\}$ is minimal it is called an \emph{atom}.  If, whenever $r\leq p$ for all $p\in\mathcal{P}$, we have $r\leq q$ for all $q\in\mathcal{Q}$ then we say $\mathcal{P}$ is \emph{below} $\mathcal{Q}$. If $p<q$, for all $p\in\mathcal{P}$ and $q\in\mathcal{Q}$, then $(\mathcal{P},\mathcal{Q})$ is a \emph{pregap}.  We say $r\in\mathbb{P}$ \emph{interpolates} the pregap $(\mathcal{P},\mathcal{Q})$ if $p<r<q$, for all $p\in\mathcal{P}$ and $q\in\mathcal{Q}$.  A \emph{gap} is a pregap with no interpolating element.  A (pre)gap $((p_n),(q_n))$ for strictly increasing $(p_n)$ and strictly decreasing $(q_n)$ is called an $(\omega,\omega)$-(pre)gap.  We say $\mathbb{P}$ is \emph{$\sigma$-closed} if every decreasing $(p_n)\subseteq\mathbb{P}\backslash\{0\}$ has a non-zero lower bound.
\end{dfn}

\begin{prp}\thlabel{rrzww}
If $\pi$ is a C$^*$-algebra homomorphism from $A$ onto $B$, $A$ has real rank zero and $\mathcal{P}(A)$ has no $(\omega,\omega)$-gaps then neither does $\mathcal{P}(B)$.
\end{prp}

\begin{proof} Given an $(\omega,\omega)$-pregap $((p_n),(q_n))$ in $\mathcal{P}(B)$, we may alternately choose $\pi$-pullbacks $(P_n)$ and $(Q_n)$ for $(p_n)$ and $(q_n)$ respectively, using \thref{S<Q<P} at each stage to ensure each $\pi$-pullback lies in between the $\pi$-pullbacks already chosen.  Then $((P_n),(Q_n))$ is an $(\omega,\omega)$-pregap and can thus be interpolated by some $P\in\mathcal{P}(A)$, from which it follows that $\pi(P)$ interpolates $((p_n),(q_n))$. \end{proof}

\begin{cor}\thlabel{pivn}
If $\pi$ is a C$^*$-algebra homomorphism from $A$ onto $B$ and $A$ is a von Neumann algebra then $\mathcal{P}(B)$ has no $(\omega,\omega)$-gaps.
\end{cor}

\begin{proof} By \thref{rrzww}, it is enough to prove that $\mathcal{P}(A)$ does not have $(\omega,\omega)$-gaps.  But an $(\omega,\omega)$-pregap $(\mathcal{P},\mathcal{Q})$ in $\mathcal{P}(A)$ corresponds to an $(\omega,\omega)$-pregap of closed subspaces, and both the intersection of the top half and the closed union of the bottom half of any such pregap will interpolate it.  The projection onto either of these subspaces will then be in $\mathcal{P}(A)$ because $A$ is a von Neumann algebra. \end{proof}

Assume we have a C$^*$-algebra $A$, $P\in\mathcal{P}(A)$ and decreasing $(P_n)\subseteq\mathcal{P}(A)$ with $||P^\perp P_n||\rightarrow0$.  Then, for any $Q\in\mathcal{P}(A)$ with $Q\leq P_n$, for all $n\in\omega$, we have \[||P^\perp Q||=||P^\perp(P_n+P^\perp_n)Q||\leq||P^\perp P_n||+||P^\perp_nQ||=||P^\perp P_n||\rightarrow0,\] and hence $Q\leq P$, i.e. $(P_n)$ is below $\{P\}$.  For \thref{cglb}, and the results that follow from it, we will need the converse of this to hold.

\begin{dfn}\thlabel{omegap}
A C$^*$-algebra $A$ has the \emph{$\omega$-property} if $||P^\perp P_n||\rightarrow0$ whenever $(P_n)\subseteq\mathcal{P}(A)$ is a decreasing sequence below $P\in\mathcal{P}(A)$.
\end{dfn}

Note that if $A$ is a C$^*$-algebra with the $\omega$-property then $\mathcal{P}(A)$ is $\sigma$-closed.  For, if not, we would have decreasing $(P_n)$ with no non-zero lower bound, i.e. $(P_n)$ would be below $0$ even though $||0^\perp P_n||=||1P_n||=||P_n||=1$, for all $n\in\omega$.

We point out here that the $\omega$-property is not necessarily invariant under homomorphisms (see the next paragraph).  Even if it were, this would not help us prove that $\mathcal{C}(H)$ has the $\omega$-property, as $\mathcal{B}(H)$ does not have the $\omega$-property for infinite dimensional Hilbert spaces $H$.  In fact, so long as $H$ is infinite dimensional, we may take orthogonal $(P_n)\subseteq\mathcal{P}(\mathcal{B}(H))\backslash\{0\}$ such that $H=\overline{\sum\mathcal{R}(P_n)}$ (for a collection of subspaces $\mathcal{V}$ of $H$, $\sum\mathcal{V}$ denotes the (possibly not closed) linear span of subspaces in $\mathcal{V}$, i.e. $\sum\mathcal{V}=\bigcup_{n\in\omega}\{\sum_{k=0}^n v_k:\forall k\leq n(v_k\in\bigcup\mathcal{V})\}$).  Then, setting $Q_n=(\sum_{k=0}^nP_k)^\perp$, for each $n\in\omega$, we see that $(Q_n)$ has no non-zero lower bound and hence $\mathcal{B}(H)$ is not even $\sigma$-closed.

However, it is not difficult to prove directly that $\mathcal{C}(H)$ has the $\omega$-property, as shown in \thref{C(H)} below.  It is also not difficult to find a homomorphism from $\mathcal{C}(H)$ to another C$^*$-algebra without the $\omega$-property, so long as $H$ has uncountable Hilbert dimension.  Specifically, let $\mathcal{C}^\infty(H)=\mathcal{B}(H)/\mathcal{K}^\infty(H)$, where $\mathcal{K}^\infty(H)=\{T\in\mathcal{B}(H):\mathcal{R}(T)\textrm{ is separable}\}$.  As $\mathcal{K}(H)\subseteq\mathcal{K}^\infty(H)$, the canonical homomorphism $\pi$ from $\mathcal{B}(H)$ to $\mathcal{C}^\infty(H)$ induces a canonical homomorphism from $\mathcal{C}(H)$ to $\mathcal{C}^\infty(H)$.  But if $(P_n)$ and $(Q_n)$ are as in the previous paragraph, this time with $\mathcal{R}(P_n)$ non-separable, for each $n\in\omega$, then $\pi(Q_n)$ has no non-zero lower bound in $\mathcal{C}^\infty(H)$ (because if we have $Q\in\mathcal{P}(\mathcal{B}(H))$ such that $\pi(Q)\leq\pi(Q_n)$, for all $n\in\omega$, then $\mathcal{R}(QP_n)$ is separable, for all $n\in\omega$, and hence $\mathcal{R}(Q)\subseteq\overline{\sum\mathcal{R}(QP_n)}$ is also separable, i.e. $\pi(Q)=0$).

We show in \thref{C(H)} below that $\mathcal{C}(H)$ has the properties relevant to this paper.  With the exception of the $\omega$-property, which as far as we know has not been defined before, these properties of $\mathcal{C}(H)$ are well known.  For example, it is shown in \cite{l} that $\mathcal{C}(H)$ has real rank zero.  In fact, it follows directly from \thref{pispec} that the spectral family approximations we require always exist in quotients of von Neumann algebras.  Indeed, the approximations are even better in this case, as we see that $E^\perp_{\pi(S)}(t)\leq\pi(E^\perp_S(t))\leq E^\perp_{\pi(S)}(t-)$, whenever $\pi$ is a homomorphism from a von Neumann algebra $A$ to a (necessarily real rank zero) C$^*$-algebra $B$, $S\in\mathcal{S}(A)$ and $t\in\mathbb{R}$.  In other words, instead of only knowing that, whenever $s\in(0,t)$, there exists $P\in\mathcal{P}(B)$ such that $E^\perp_{\pi(S)}(t)\leq P\leq E^\perp_{\pi(S)}(s)$, we in fact have just one $P$, namely $\pi(E^\perp_S(t))$, for which $E^\perp_{\pi(S)}(t)\leq P\leq E^\perp_{\pi(S)}(s)$ for all $s\in(0,t)$.  Such C$^*$-algebras might be considered as the non-commutative analogs of extremally disconnected topological spaces, just as real rank zero C$^*$-algebras are considered as the non-commutative analogs of totally disconnected spaces.

\begin{thm}\thlabel{C(H)}
If $H$ is an infinite dimensional Hilbert space then $\mathcal{C}(H)$ has real rank zero and the $\omega$-property, and $\mathcal{P}(\mathcal{C}(H))$ has no atoms or $(\omega,\omega)$-gaps.
\end{thm}

\begin{proof} The first statement follows from \thref{pirrz} and the last follows from \thref{pivn}.  To see that $\mathcal{P}(\mathcal{C}(H))$ has no atoms simply note that $\pi(P)\neq 0$ if and only if $\mathcal{R}(P)$ is infinite dimensional, and every infinite dimensional closed subspace contains a closed infinite dimensional infinite codimensional subspace. To see that $\mathcal{C}(H)$ has the $\omega$-property, take $p,(p_n)\subseteq\mathcal{P}(\mathcal{C}(H))$ with $(p_n)$ strictly decreasing and $||p^\perp p_n||>\epsilon>0$, for all $n\in\omega$.  By \thref{Q<P}, we can find $P,(P_n)\subseteq\mathcal{P}(\mathcal{B}(H))$ such that $(P_n)$ is decreasing, $\pi(P)=p$ and $\pi(P_n)=p_n$, for all $n\in\omega$.  Then we may recursively choose an orthonormal sequence $(v_n)$ (i.e. a sequence of orthogonal unit vectors) such that $v_n\in\mathcal{R}(P_n)$ and $||P^\perp v_n||>\epsilon$, for all $n\in\omega$.  Orthonormality implies that $(v_n)$ converges weakly to $0$, and hence $(Kv_n)$ converges in the norm topology to $0$, for all $K\in\mathcal{K}(H)$ (see \cite{a} Theorem 6.3).  Thus, for all such $K$, $||P^\perp Q-K||\geq\lim\sup||P^\perp v_n||$, where $Q$ is the projection onto $\overline{\mathrm{span}}(v_n)$.  This means $||p^\perp\pi(Q)||\geq\epsilon$ and hence $\pi(Q)\nleq p$.  But also the codimension of $\mathcal{R}(Q)\cap\mathcal{R}(P_n)$ in $\mathcal{R}(Q)$ is finite$(=n)$, for all $n\in\omega$, and hence $\pi(Q)\leq p_n$, for all $n\in\omega$. \end{proof}

\section{Lower Bounds (Finite Collections)}\label{lb}

In this section we take a C$^*$-algebra $A$, faithfully represented on a Hilbert space $H$ (so $A\subseteq\mathcal{B}(H)$), and investigate the lower bounds of subsets of $\mathcal{P}(A)$.  For $\mathcal{P}\subseteq\mathcal{P}(\mathcal{B}(H))$, $\bigwedge\mathcal{P}$ denotes the greatest lower bound of $\mathcal{P}$ in $\mathcal{P}(\mathcal{B}(H))$, so $\mathcal{R}(\bigwedge\mathcal{P})=\bigcap_{P\in\mathcal{P}}\mathcal{R}(P)$.  Also $\bigvee\mathcal{P}$ denotes the least upper bound of $\mathcal{P}$ in $\mathcal{P}(\mathcal{B}(H))$, so $\mathcal{R}(\bigvee\mathcal{P})=\overline{\sum_{P\in\mathcal{P}}\mathcal{R}(P)}$.

We first prove one general theorem, using the spectrum to characterize the points at which a (algebraic) homomorphism is a lattice homomorphism.  This might appear a bit intimidating at first, but afterwards we give examples and applications for a number of important special cases.

\begin{thm}\thlabel{finglbpi}
Let $\pi$ be a $C^*$-algebra homomorphism from $A\subseteq\mathcal{B}(H)$ to $B\subseteq\mathcal{B}(H')$, take $P_0,\ldots,P_n\in\mathcal{P}(A)$ and set $T=P_0\ldots P_n$, $P=P_0\wedge\ldots\wedge P_n$ and $p=\pi(P_0)\wedge\ldots\wedge\pi(P_n)$.  If $\sup(\sigma(T^*T)\backslash\{1\})<1$ then $P\in A$ and $\pi(P)=p$.  Further assume that $A$ has real rank zero, $\mathcal{P}(B)$ is $\sigma$-closed and $\pi^{-1}[\{0\}]\subseteq\mathcal{K}(H)$.  If $\pi(R)$ is a g.l.b. of $\pi(P_0),\ldots,\pi(P_n)$ in $\mathcal{P}(B)$, for some $R\in\mathcal{P}(A)$ below $P_0,\ldots,P_n$, then $\sup(\sigma(T^*T)\backslash\{1\})<1$.
\end{thm}

\begin{proof}
If $s=\sup(\sigma(T^*T)\backslash\{1\})<1$ then we may let $f$ be a continuous function $f:\mathbb{R}\rightarrow[0,1]$ such that $f(t)=0$, for all $t\leq s$, and $f(t)=1$, for all $t\geq1$.  Then $P=E^\perp_{T^*T}(1-)=f(T^*T)\in A$ and $\pi(P)=\pi(f(T^*T))=f(\pi(T^*T))=E_{\pi(T^*T)}(1-)=p$ (using the fact that $\sigma(\pi(T^*T))\subseteq\sigma(T^*T)$).

Assume further that $A$ has real rank zero, $\mathcal{P}(B)$ is $\sigma$-closed and $\pi^{-1}[\{0\}]\subseteq\mathcal{K}(H)$.  Take $t_m\uparrow 1$ and $(Q_m)\subseteq\mathcal{P}(A)$ with $E^\perp_{T^*T}(t_{m+1})\leq Q_m\leq E^\perp_{T^*T}(t_m)$, for all $m\in\omega$, so we have $\bigwedge_{m\in\omega}Q_m=E^\perp_{T^*T}(1-)=P$.  If $\sup(\sigma(T^*T)\backslash\{1\})=1$ then $Q_m-P$ must have infinite rank, for all $m\in\omega$, and hence the same must be true for $Q_m-R$, for any $R\in\mathcal{P}(A)$ with $R\leq P$.  As $\pi^{-1}[\{0\}]\subseteq\mathcal{K}(H)$, we have $\pi(Q_m)>\pi(R)$, for all $m\in\omega$.  Thus $(\pi(Q_m)-\pi(R))$ must be bounded below by some $q\in\mathcal{P}(B)\backslash\{0\}$, yielding $p\geq\pi(R)+q>\pi(R)$ and $\pi(R)+q\in\mathcal{P}(B)$, i.e. $\pi(R)$ could not be the g.l.b. of $\pi(P_0),\ldots,\pi(P_n)$ in $\mathcal{P}(B)$.
\end{proof}

When $A=\mathcal{B}(H)$, $B\approx\mathcal{C}(H)$ and $\pi$ is canonical, \thref{finglbpi} tells us that, for $P,Q\in\mathcal{P}(A)$,
\begin{eqnarray}\label{piPQ}
&& \exists R\in\mathcal{P}(A)\textrm{ below both $P$ and $Q$ such that $\pi(R)$ is a g.l.b. of $\pi(P)$ and $\pi(Q)$ in $\mathcal{P}(B)$}\nonumber\\
&\Leftrightarrow& \pi(P\wedge Q)=\pi(P)\wedge\pi(Q)\nonumber\\
&\Leftrightarrow& \sigma(PQ)\backslash\{1\}<1.\nonumber
\end{eqnarray}
Note that when $\pi(P)\wedge\pi(Q)\notin\mathcal{P}(B)$, we certainly can not have $\pi(P\wedge Q)=\pi(P)\wedge\pi(Q)$.  However, even when $\pi(P)\wedge\pi(Q)\in\mathcal{P}(B)$, there is still no guarantee that $\pi(P\wedge Q)=\pi(P)\wedge\pi(Q)$.  For example, say $(e_n)$ is an orthonormal basis for $H$ and consider $P,Q\in\mathcal{P}(\mathcal{B}(H))$ with
\begin{equation}\label{badPQ}
\mathcal{R}(P)=\overline{\mathrm{span}}(e_{2n})\qquad\textrm{and}\qquad\mathcal{R}(Q)=\overline{\mathrm{span}}(e_{2n}+\frac{1}{n+1}e_{2n+1}).
\end{equation}
Then $\pi(P)=\pi(Q)$ is certainly in $B$, even though it is not equal to $0=\pi(0)=\pi(P\wedge Q)$.

When $A=B\approx\mathcal{C}(H)$ and $\pi$ is the identity, \thref{finglbpi} tells us that, for $p,q\in\mathcal{P}(A)$,
\begin{equation}\label{pq}
p\textrm{ and }q\textrm{ have a g.l.b. in }\mathcal{P}(A)\quad\Leftrightarrow\quad p\wedge q\in A\quad\Leftrightarrow\quad\sigma(pq)\backslash\{1\}<1.
\end{equation}
Incidentally, until recently it seems to have been implicitly assumed in a number of previous papers that the projections in the Calkin algebra are a lattice.  The first example of a pair of projections in the Calkin algebra that do not have a g.l.b. was provided by Weaver and can be found in \cite{b} Proposition 5.26.  For another (very similar) example, consider the pair $p=\pi(P^\omega)$ and $q=\pi(Q^\omega)$, where $P^\omega$ and $Q^\omega$ are the projections on $\bigoplus_\omega H\cong H$ composed of countably many copies of the $P$ and $Q$ in (\ref{badPQ}) (and $\pi:\mathcal{B}(H)\rightarrow\mathcal{C}(H)$ is canonical).  Then $\sigma(pq)\backslash\{1\}=\{\frac{n^2}{n^2+1}:n\in\omega\}$, and hence $p$ and $q$ have no g.l.b. in $\mathcal{C}(H)$ by (\ref{pq}).

Next we characterize g.l.b.s of finite subsets of projections using a simple norm expression.

\begin{thm}\thlabel{glbf}
Let $A$ be a C$^*$-algebra, take $P_0,\ldots,P_n\in\mathcal{P}(A)$ and set $T=P_0P_1\ldots P_n$ and $P=P_0\wedge\ldots\wedge P_n$.  If $R\leq P$ and $||T-R||<1$ then $R=P$.
Moreover, assuming $A$ has real rank zero and $\mathcal{P}(A)$ is $\sigma$-closed, if $R\in\mathcal{P}(A)$ is the g.l.b. of $P_0,\ldots,P_n$ in $\mathcal{P}(A)$ then $||T-R||<1$.
\end{thm}

\begin{proof}
As $P\leq P_k$ and hence by definition $PP_k=P$, for $k\leq n$, we have $PT=P$.  If $R\leq P$ then also $R=R^*=TR$ and hence $TR^\perp=T-TR=T-R$ so $PR^\perp=PTR^\perp=P(T-R)$.  If we actually have $R<P$ then $PR^\perp$ is a non-zero projection so $1=||P(T-R)||\leq||T-R||$, which proves the first part.

Now assume further that $A$ has real rank zero and $\mathcal{P}(A)$ is $\sigma$-closed.  If $R\in\mathcal{P}(A)$ is below $P_0,\ldots,P_n$ then $(T-R)^*(T-R)=T^*TR^\perp$.  If $R$ is actually a g.l.b. of $P_0,\ldots,P_n$ in $\mathcal{P}(A)$ then, by \thref{finglbpi} (with $\pi$ the identity), $\sup(\sigma(T^*T)\backslash\{1\})<1$ and $P\in A$.  Thus we must actually have $R=P=E^\perp_{T^*T}(1-)$ and $||T-R||^2=||T^*TE_{T^*T}(1-)||<1$ and hence $||T-R||<1$.
\end{proof}

When $A\approx\mathcal{C}(H)$, \thref{glbf} tells us that, for $P,Q,R\in\mathcal{P}(A)$,
\begin{equation}\label{PQ-R}
\textrm{$R$ is the g.l.b. of $P$ and $Q$ in }\mathcal{P}(A)\quad\Leftrightarrow\quad R=P\wedge Q\quad\Leftrightarrow\quad R\leq P,Q\textrm{ and }||PQ-R||<1.
\end{equation}
The most important case of \thref{glbf} occurs when $R=0$, in which case (\ref{PQ-R}) becomes
\begin{equation}\label{PQ}
\textrm{$P$ and $Q$ are bounded below in }\mathcal{P}(A)\backslash\{0\}\quad\Leftrightarrow\quad P\wedge Q\neq 0\quad\Leftrightarrow\quad||PQ||=1.
\end{equation}
Note that is crucial, for $\Leftarrow$ part of both equivalences in (\ref{PQ}), that $\mathcal{P}(A)$ be $\sigma$-closed.  For consider the $P$ and $Q$ in (\ref{badPQ}).  For a counterexample to the second $\Leftarrow$, simply note that $P\wedge Q=0$ even though $||PQ||=1$ (and $\mathcal{B}(H)$ has real rank zero).  For a counterexample to the first $\Leftarrow$, let $\pi_u$ be the universal representation of $\mathcal{B}(H)$ on $H_u$ and let $\pi:\mathcal{B}(H)\rightarrow\mathcal{C}(H)$ be canonical.  Then $\pi\circ\pi^{-1}_u$ can be extended to a normal representation $\theta$ on the entirety of $\pi_u[\mathcal{B}(H)]''$ and hence $\theta(\pi_u(P)\wedge\pi_u(Q))=\theta(\pi_u(P))\wedge\theta(\pi_u(Q))=\pi(P)\wedge\pi(Q)=\pi(P)\neq0_u$.  Thus $\pi_u(P)\wedge\pi_u(Q)\neq0_u$ even though $0_u$ is a g.l.b. of $\pi_u(P)$ and $\pi_u(Q)$ in $\mathcal{P}(\pi_u[\mathcal{B}(H)])$ (because $P\wedge Q=0$).

As an immediate application of \thref{glbf} for $R=0$, we note that it implies the collection of projections on which a state (i.e. a positive functional on $A$ of norm $1$) is $1$ is centred (in the usual order theoretic sense, i.e. that every finite subset has a non-zero lower bound)\footnote{In fact, a lot more can be said about centred sets and states.  For example, for every centred $\mathcal{P}\subseteq\mathcal{P}(A)$ there exists a state which is $1$ on all of $\mathcal{P}$, and this even yields a bijective correspondence between pure states and maximal centred subsets of $\mathcal{P}(A)$ \textendash\, see \cite{q} for these and other related results.}.

\begin{cor}\thlabel{centred}
Let $\phi$ be a state on C$^*$-algebra $A$ of real rank zero and assume further that $\mathcal{P}(A)$ is $\sigma$-closed.  Then $\mathcal{P}(\phi)=\{P\in\mathcal{P}(A):\phi(P)=1\}$ is centred.
\end{cor}

\begin{proof} By the Gelfand-Naimark-Segal construction, there exists a representation $\pi$ on a Hilbert space $H$ and $v\in H$ such that $\phi(S)=\langle\pi(S)v,v\rangle$, for all $S\in A$ (see \cite{m} Theorem 1.6.3).  Thus if $P\in\mathcal{P}(A)$ satisfies $1=\phi(P)=\langle\pi(P)v,v\rangle$ then $\pi(P)v=v$.  So, for any $P_0,\ldots,P_n\in\mathcal{P}(\phi)$, setting $T=P_0\ldots P_n$ we have $\pi(T)v=v$ and hence $||T||\geq\phi(T)=\langle\pi(T)v,v\rangle=\langle v,v\rangle=1$. \end{proof}

Lastly, we show that, even though $\mathcal{P}(A)$ may not be a lattice, it is still relatively well-behaved.  Specifically, it is separative, i.e. whenever $p\nleq q$ there exists non-zero $r\leq q$ with $r\wedge q=0$.

\begin{thm}\thlabel{sep}
If $A$ is a C$^*$-algebra of real rank zero then $\mathcal{P}(A)$ is separative.
\end{thm}

\begin{proof} Assume $P,Q\in\mathcal{P}(A)$ satisfy $P\nleq Q$, so $||Q^\perp P||>0$.  Take $s,t\in\mathbb{R}$ and $R\in\mathcal{P}(A)$ with $0<s<t<||Q^\perp P||^2$ and $E^\perp_{PQ^\perp P}(t)\leq R\leq E^\perp_{PQ^\perp P}(s)$.  Note that $P^\perp\leq E_{PQ^\perp P}(0)$ and hence $P\geq E^\perp_{PQ^\perp P}(0)\geq E^\perp_{PQ^\perp P}(s)\geq R$.  Thus $||QR||<1$ because, for $v\in\mathcal{R}(R)\subseteq\mathcal{R}(P)$,
\[||Qv||=\sqrt{1-||Q^\perp v||^2}=\sqrt{1-||Q^\perp Pv||^2}=\sqrt{1-\langle PQ^\perp Pv,v\rangle}\leq\sqrt{1-s}.\]
By \thref{glbf} (see (\ref{PQ})), $Q\wedge R=0$.  But $t<||Q^\perp P||^2=||PQ^\perp P||$ so $0<E^\perp_{PQ^\perp P}(t)\leq R$.
\end{proof}

\section{Lower Bounds (Countable Collections)}

Say we are given an aribtrary countable subset $(p_n)$ of a lattice $\mathbb{P}$.  Setting $q_n=p_0\wedge\ldots\wedge p_n$, for all $n\in\omega$, we see that $(q_n)$ is a (possibly non-strictly) decreasing sequence with exactly the same lower bounds.  Even though $\mathcal{P}(A)$ may not be a lattice, for C$^*$-algebras $A$ of real rank zero, it will still have this property.  Indeed, we already saw in the proof of \thref{finglbpi} that, for every finite subset of $\mathcal{P}(A)$, there exists a decreasing sequence with exactly the same lower bounds.  We now show that, if this holds even for just two element subsets $\mathcal{P}$ of $\mathcal{P}(A)$ then, so long as $A$ has the $\omega$-property, it must also hold for arbitrary countable $\mathcal{P}\subseteq\mathcal{P}(A)$.  Note that this theorem applies even to C$^*$-algebras that do not have real rank zero, so long as there is some other method available to prove the two element case.

\begin{thm}\thlabel{cglb}
Let $A$ be an arbitrary C$^*$-algebra with the $\omega$-property.  If, for every two-element set $\mathcal{P}\subseteq\mathcal{P}(A)$, there exists decreasing $(Q_n)\subseteq\mathcal{P}(A)$ with $\bigwedge\mathcal{P}=\bigwedge Q_n$ then the same is true for countable $\mathcal{P}$.  Moreover, if $\mathcal{P}\subseteq\mathcal{P}(A)$ has a g.l.b. $Q\in\mathcal{P}(A)$ then $Q=\bigwedge\mathcal{F}$ for some finite $\mathcal{F}\subseteq\mathcal{P}$.
\end{thm}

\begin{proof} We first prove the theorem for finite $\mathcal{P}$, by induction on the number of elements in $\mathcal{P}$.  For the induction step, we have decreasing $(P_n)\subseteq\mathcal{P}(A)$ and $P\in\mathcal{P}(A)$ and we have to find decreasing $(Q_n)\subseteq\mathcal{P}(A)$ with $\bigwedge Q_n=P\wedge\bigwedge P_n$.  For each $n\in\omega$, let $(P_{n,m})_{m\in\omega}$ be a decreasing sequence with $P\wedge P_n=\bigwedge_{m\in\omega}P_{n,m}$.  Recursively define $(Q_n)\subseteq\mathcal{P}(A)$ such that $Q_n\geq P\wedge P_n$, for each $n\in\omega$, as follows.  Set $Q_0=P_{0,0}$ and, for each $n\in\omega$, let $m_n$ be such that $||P^\perp_{k,n}P_{n,m_n}||\leq2^{-n}$, for all $k<n$, and $||Q^\perp_{n-1}P_{n,m_n}||\leq2^{-n}$, which is possible because $A$ has the $\omega$-property.  Then let
\begin{equation}\label{Qn}
Q_n=E^\perp_{Q_{n-1}P_{n,m_n}Q_{n-1}}((1-2^{-2n})-),
\end{equation}
noting that $\sigma(Q_{n-1}P_{n,m_n}Q_{n-1})\cap(0,1-2^{-2n})$ because $||Q^\perp_{n-1}P_{n,m_n}||\leq2^{-n}$ (apply (\ref{||PQ||}), (\ref{sigmaPQP}) and (\ref{sigmaPQperpP})),\footnote{Incidentally, this means $Q_n=E^\perp_{Q_{n-1}P_{n,m_n}Q_{n-1}}(0)$ so $\mathcal{R}(Q_n)=\mathcal{R}(Q_{n-1}P_{n,m_n}Q_{n-1})=\mathcal{R}(Q_{n-1}P_{n,m_n})$.} and hence $Q_n=f(Q_{n-1}P_{n,m_n}Q_{n-1})\in A$, where $f$ is any continuous function which is $0$ at $0$ and $1$ on $[1-2^{-2n},1]$.  Note that the recursion may continue because \[Q_n\geq E^\perp_{Q_{n-1}P_{n,m_n}Q_{n-1}}(1-)=Q_{n-1}\wedge P_{n,m_n}\geq(P\wedge P_{n-1})\wedge(P\wedge P_n)=P\wedge P_n.\]

As $Q_n\geq P\wedge P_n$, for each $n\in\omega$, $\bigwedge Q_n\geq P\wedge\bigwedge P_n$.  On the other hand, for all $n\in\omega$ and $k<n$, we have $||P^\perp_{k,n}Q_n||\leq||P^\perp_{k,n}P_{n,m_n}||+||P_{n,m_n}^\perp Q_n||\leq2^{1-n}$ (by (\ref{Qn})) so $\bigwedge Q_n\leq P\wedge\bigwedge P_n$.

Thus the finite case is done, and to prove the countably infinite case let $(P_n)$ enumerate $\mathcal{P}$ and, for each $n\in\omega$, let $(P_{n,m})_{m\in\omega}$ be a decreasing sequence with $\bigwedge_mP_{n,m}=P_0\wedge\ldots\wedge P_n$.  Then define $(Q_n)$ exactly as before.  For the last statement, note that if $\mathcal{P}$ has a g.l.b. $P$ then, as $\mathcal{P}(A)$ is $\sigma$-closed, $(Q_n)$ must be constant from some point onwards, i.e. $Q_n=Q_m$ for all $n\geq m$, for some $m\in\omega$, and $Q_m$ is therefore the g.l.b. of $\mathcal{P}$.  In particular, $Q_m\leq P_0\wedge\ldots\wedge P_m$.  But the definition of the $(Q_n)$ shows that also $P_0\wedge\ldots\wedge P_m\leq Q_m$. \end{proof}

\thref{cglb} (minus the last sentence) can also be proved for arbitrary C$^*$-algebras of real rank zero, whether they have the $\omega$-property or not.  Unfortunately, the proof is not so simple in this case, as it relies on some slightly technical calculations involving spectral families.  Before we prove this in \thref{arbcount}, let us first prove two lemmas.  \thref{nicecount} essentially says that if a countable sequence of projections is close enough to being a decreasing sequence then there necessarily exists a decreasing sequence with the same lower bounds.

\begin{lem}\thlabel{nicecount}
Let $A$ be an arbitrary C$^*$-algebra and assume $(P_n)\subseteq\mathcal{P}(A)$ satisfies $||P^\perp_nP_{n+1}||<1$, for all $n\in\omega$, and
\begin{equation}\label{techcon}
\sum_{k<n}||P^\perp_kP_{k+1}\ldots P_n||\rightarrow0,\quad\textrm{as}\ n\rightarrow\infty.
\end{equation}
Then there exists (possibly non-strictly) decreasing $(Q_n)\subseteq\mathcal{P}(A)$ such that $\bigwedge P_n=\bigwedge Q_n$.
\end{lem}

\begin{proof} For all $n\in\omega$, $||P^\perp_nP_{n+1}||<1$ so, for any $P\leq P_{n+1}$, we have \[\mathcal{R}(P_nP)=\mathcal{R}(E^\perp_{P_nPP_n}(0))=\mathcal{R}(E^\perp_{P_nPP_n}(t-)),\] where $t=1-||P^\perp_nP_{n+1}||^2$.  It follows, by the continuous functional calculus, that the projection onto $\mathcal{R}(P_nP)$ is in $A$ (see (\ref{Qn}) and the comments immediately after).  Thus we have $(Q_n)\subseteq\mathcal{P}(A)$ such that $\mathcal{R}(Q_n)=\mathcal{R}(P_0\ldots P_n)$, for all $n\in\omega$.

For $m,n\in\omega$, with $m\leq n$, and $v\in\mathcal{R}(P_n)$, $||v-Q_nv||\leq||v-P_0\ldots P_{n-1}v||$ and hence \begin{equation}\label{eq2}
||Q^\perp_mP_n||\leq||Q^\perp_nP_n||\leq||P_n-P_0\ldots P_n||\leq\sum_{k<n}||P^\perp_kP_{k+1}\ldots P_n||\rightarrow0\textrm{ as }n\rightarrow\infty.
\end{equation}
Thus $\bigwedge P_n\leq\bigwedge Q_n$.

On the other hand, again for $m,n\in\omega$, with $m\leq n$, and $v\in\mathcal{R}(P_n)$, \[||P^\perp_mP_0\ldots P_{n-1}v||\leq||P_0\ldots P_{n-1}v-P_m\ldots P_{n-1}v||\leq||\sum_{k<m}P^\perp_kP_{k+1}\ldots P_{n-1}v||,\] and hence, as long as $||P_n-P_0\ldots P_n||<1$, which is true for sufficiently large $n$ by (\ref{eq2}),
\begin{eqnarray*}
||P^\perp_mQ_n|| &=& \sup_{v\in\mathcal{R}(P_n)}||P^\perp_mP_0\ldots P_{n-1}v||/||P_0\ldots P_{n-1}v||\\
&\leq& \sup_{v\in\mathcal{R}(P_n)}||\sum_{k<m}P^\perp_kP_{k+1}\ldots P_{n-1}v||/(||v||-||v-P_0\ldots P_{n-1}v||)\\
&\leq& ||\sum_{k<m}P^\perp_kP_{k+1}\ldots P_n||/(1-||P_n-P_0\ldots P_n||)\\
&\leq& \sum_{k<n}||P^\perp_kP_{k+1}\ldots P_n||/(1-||P_n-P_0\ldots P_n||)\\
&\rightarrow& 0\quad\textrm{ as }n\rightarrow\infty\qquad\textrm{by (\ref{eq2})}.
\end{eqnarray*}
Thus $\bigwedge Q_n\leq\bigwedge P_n$. \end{proof}

This next lemma provides the spectral family inequality required in the proof of \thref{arbcount}.

\begin{lem}
For Hilbert space $H$, $P\in\mathcal{P}(\mathcal{B}(H))$, $S\in\mathcal{S}(\mathcal{B}(H))$ and $s,t<||S||$ with $s\geq0$,
\begin{equation}\label{EE}
||E_S(t)E^\perp_{PSP}(s)||^2\leq(||S||-s)/(||S||-t).
\end{equation}
\end{lem}

\begin{proof}  For all $v\in\mathcal{R}(E^\perp_{PSP}(s))$,
\begin{eqnarray*}
s||v||^2 &\leq& \langle PSPv,v\rangle\\
&=& \langle Sv,v\rangle\\
&=& \langle SE_S(t)v,v\rangle+\langle SE^\perp_S(t)v,v\rangle\\
&\leq& t\langle E_S(t)v,v\rangle+||S||\langle E^\perp_S(t)v,v\rangle\\
&=& t||E_S(t)v||^2+||S||(||v||^2-||E_S(t)v||^2),\\
\textrm{and thus}\quad(||S||-t)||E_S(t)v||^2 &\leq& (||S||-s)||v||^2,
\end{eqnarray*}
which immediately yields (\ref{EE}). \end{proof}

\begin{thm}\thlabel{arbcount}
Let $A$ be a C$^*$-algebra of real rank zero.  Then, for any $(P_n)\subseteq\mathcal{P}(A)$, there exists (possibly non-strictly) decreasing $(Q_n)\subseteq\mathcal{P}(A)$ such that $\bigwedge P_n=\bigwedge Q_n$.  Moreover, if $\mathcal{P}(A)$ is $\sigma$-closed\footnote{If $A$ is a C$^*$-algebra of real rank zero and $\mathcal{P}(A)$ is $\sigma$-closed then $A$ must in fact have the $\omega$-property (see \cite{s} Theorem 6.3 and the comment after) and so the theorem can also be derived from \thref{cglb} in this case.} and $(P_n)$ has a g.l.b. $Q\in\mathcal{P}(A)$ then $Q=P_0\wedge\ldots\wedge P_m$, for some $m\in\omega$.
\end{thm}

\begin{proof} By \thref{nicecount}, it suffices to find $(Q_n)\subseteq\mathcal{P}(A)$ such that $||Q^\perp_nQ_{n+1}||<1$, for all $n\in\omega$, $\bigwedge P_n=\bigwedge Q_n$ and $\sum_{k<n}||P^\perp_kP_{k+1}\ldots P_n||\rightarrow0$, as $n\rightarrow\infty$.  To this end, for each $n\in\omega$, let $T_n=P_0\ldots P_n$ and $E_n=E_{T_n^*T_n}$.  Set $t_{0,n}=0$, for all $n\in\omega$, and, once $(t_{k,n})$ has been defined for all $k<m$ and $n\in\omega$, set $t_{m,n}=1-(1-t_{m-1,n+1})/(m+n+1)^6$, for all $n\in\omega$.  After this recursion is complete, take $(Q_n)\subseteq\mathcal{P}(A)$ such that $E^\perp_n(t_{n,1})\leq Q_n\leq E^\perp_n(t_{n,0})$, for all $n\in\omega$ (so $Q_0=P_0$).  We certainly have $\bigwedge P_n\leq\bigwedge Q_n$.  On the other hand, by (\ref{EE}) and our choice of $(t_{m,n})$,
\begin{equation}\label{teq}
||E_k(t_{k,n-k})E^\perp_{k+1}(t_{k+1,n-k-1})||\leq\sqrt{(1-t_{k+1,n-k-1})/(1-t_{k,n-k})}=1/(n+1)^3,
\end{equation}
for all $k\in n$.  Thus, for all $m,n\in\omega$ with $n>m$,
\begin{eqnarray*}
||P^\perp_mQ_n|| &=& ||P^\perp_mE^\perp_n(t_{n,0})||\\
&=& ||E^\perp_n(t_{n,0})-P_mE^\perp_n(t_{n,0})||\\
&\leq& ||E^\perp_n(t_{n,0})-E^\perp_m(t_{m,n-m})E^\perp_{m+1}(t_{m+1,n-m-1})\ldots E^\perp_n(t_{n,0})||\\
&\leq& \sum_{m\leq k<n}||E_k(t_{k,n-k})E^\perp_{k+1}(t_{k+1,n-k-1})E^\perp_{k+2}(t_{k+2,n-k-2})\ldots E^\perp_n(t_{n,0})||\\
&\leq& \sum_{m\leq k<n}||E_k(t_{k,n-k})E^\perp_{k+1}(t_{k+1,n-k-1})||\\
&\leq& \sum_{m\leq k<n}1/(n+1)^3\qquad\textrm{by (\ref{teq})}\\
&\leq& 1/(n+1)^2\\
&\rightarrow& 0.
\end{eqnarray*}
Therefore, we also have $\bigwedge Q_n\leq\bigwedge P_n$ and hence $\bigwedge Q_n=\bigwedge P_n$.

Also, again for all $m,n\in\omega$ with $n>m$,
\begin{eqnarray*}
&& ||Q^\perp_mQ_{m+1}\ldots Q_n||\\
&\leq& ||Q^\perp_mE^\perp_{m+1}(t_{m+1,n-m-1})E^\perp_{m+2}(t_{m+2,n-m-2})\ldots E^\perp_n(t_{n,0})||\\
&& +\sum_{m<k<n}||Q^\perp_mQ_{m+1}\ldots Q_{k-1}(Q_k-E^\perp_k(t_{k,n-k}))E^\perp_{k+1}(t_{k+1,n-k-1})\ldots E^\perp_n(t_{n,0})||\\
&\leq& ||E_m(t_{m,n-m})E^\perp_{m+1}(t_{m+1,n-m-1})||+\sum_{m<k<n}||E_k(t_{k,n-k})E^\perp_{k+1}(t_{k+1,n-k-1})||\\
&=& \sum_{m\leq k<n}||E_k(t_{k,n-k})E^\perp_{k+1}(t_{k+1,n-k-1})||\\
&\leq& 1/(n+1)^2\qquad\textrm{by (\ref{teq})}.\\
\end{eqnarray*}
Therefore, for all $n\in\omega$,
\begin{eqnarray*}
\sum_{0<m<n}||Q^\perp_mQ_{m+1}\ldots Q_n|| &\leq& \sum_{m<n}1/(n+1)^2\\
&\leq& 1/(n+1)\\
&\rightarrow& 0.
\end{eqnarray*}

Thus the first statement has been proved and the last statement follows exactly as in the proof of the last statement in \thref{cglb}. \end{proof}

For an application of this, we again turn to states, and show that the centred sets mentioned in \thref{centred} are even \emph{countably} centred, i.e. even countable subsets have non-zero lower bounds.

\begin{cor}
Let $\phi$ be a state on C$^*$-algebra $A$ of real rank zero and assume further that $\mathcal{P}(A)$ is $\sigma$-closed.  Then $\mathcal{P}(\phi)=\{P\in\mathcal{P}(A):\phi(P)=1\}$ is countably centred.
\end{cor}

\begin{proof} For $\mathcal{P}\subseteq\mathcal{P}(\phi)$, having no non-zero lower bound is equivalent to saying $0$ is a g.l.b. of $\mathcal{P}$ in $\mathcal{P}(A)$.  If this were true for some countable $\mathcal{P}$ then it would also be true for some finite $\mathcal{P}$, by \thref{arbcount}, contradicting \thref{centred}. \end{proof}

The next little result shows that, in the order structures we are dealing with, if one side of a gap is a singleton, it actually dominates its side of the gap.

\begin{lem}\thlabel{gaplub}
Let $A$ be a C$^*$-algebra of real rank zero such that $\mathcal{P}(A)$ has no atoms.  If $(\mathcal{P},\{P\})$ is a gap for some $P\in\mathcal{P}(A)$ and $\mathcal{P}\subseteq\mathcal{P}(A)$ then $P$ is the l.u.b. of $\mathcal{P}$ in $\mathcal{P}(A)$.
\end{lem}

\begin{proof} Assume $P$ is not the l.u.b.\ of $\mathcal{P}$ in $\mathcal{P}(A)$.  Thus there exists $Q\in\mathcal{P}(A)$ such that $(\mathcal{P},\{Q\})$ is a pregap but $P\nleq Q$.  Thus $E^\perp_{PQP}(1-)\neq P$ and hence there exists $s\in(0,1)$ such that $E^\perp_{PQP}(s)\neq P$.  Take any $t\in(s,1)$ and let $R\in\mathcal{P}(A)$ be such that $E^\perp_{PQP}(t)\leq R\leq E^\perp_{PQP}(s)$.  Then we have $R<P$ which, as $P-R$ is not an atom, means we have $S\in\mathcal{P}(A)$ such that $P\wedge Q\leq R<S<P$.  Hence $S$ interpolates $(\mathcal{P},\{P\})$, contradicting the fact it is a gap. \end{proof}

We can now show that C$^*$-algebras like the Calkin algebra have no non-trivial countable gaps.

\begin{thm}\thlabel{nogap}
Let $A$ be a C$^*$-algebra of real rank zero such that $\mathcal{P}(A)$ has no atoms or $(\omega,\omega)$-gaps.  If $(\mathcal{P},\mathcal{Q})$ is a gap for (non-empty) countable $\mathcal{P},\mathcal{Q}\subseteq\mathcal{P}(A)$ then there exists $Q\in\mathcal{P}\cup\mathcal{Q}$ that is both an l.u.b. of $\mathcal{P}$ and g.l.b. of $\mathcal{Q}$ in $\mathcal{P}(A)$.
\end{thm}

\begin{proof} By \thref{arbcount}, there exists increasing $(P_n)$ and decreasing $(Q_n)$ with the same upper and lower bounds as $\mathcal{P}$ and $\mathcal{Q}$ respectively.  Thus $(P_n)$ or $(Q_n)$ must eventually be constant because $\mathcal{P}(A)$ has no $(\omega,\omega)$-gaps.  Without loss of generality, assume that $Q_n=Q$ from some point onwards, and hence $Q$ is the g.l.b. of $\mathcal{Q}$ in $\mathcal{P}(A)$.  By \thref{gaplub}, $Q$ is also the l.u.b. of $\mathcal{P}$, and must also be in $\mathcal{P}\cup\mathcal{Q}$, otherwise it would interpolate $(\mathcal{P},\mathcal{Q})$ and this would not be a gap. \end{proof}

If $A=\mathcal{C}(H)$ and both $\mathcal{P}$ and $\mathcal{Q}$ are pairwise \emph{commutative} then the above result follows from \cite{b} Lemma 5.34.  In particular \cite{b} Lemma 5.34 provides an alternative way of proving that $\mathcal{P}(\mathcal{C}(H))$ has no $(\omega,\omega)$-gaps.

Yet again, let us show how this result tells us something interesting about states.  Say we have a state $\phi$ on a C$^*$-algebra like the Calkin algebra, together with a pair of projections $P$ and $Q$ in the algebra such that $\phi(P)=1=\phi(Q)$ even though $\phi(R)<1$ for every projection $R$ below both $P$ and $Q$.  It is probably a little counter-intuitive that such a situation is even possible\footnote{ but it is, in fact if $\phi$ is the state on the Calkin algebra given by $\phi(\pi(T))=\phi_\mathcal{U}(T)=\lim_{n\rightarrow\mathcal{U}}\langle Te_n,e_n\rangle$, where $e_n$ is an orthonormal basis and $\mathcal{U}$ is a non-principal non-P-point ultrafilter, then there will even exist projections $P$ and $Q$ such that $\phi(P)=1=\phi(Q)$ despite the fact $\phi(R)=0$, for all projections $R\leq P,Q$.  Moreover, it is even consistent with ZFC that for \emph{every} pure state $\phi$ on the Calkin algebra, there exist projections $P$ and $Q$ such that $\phi(P)=1=\phi(Q)$ even though $\phi(R)<1$, for all projections $R\leq P,Q$ \textendash\, see \cite{q}.}, although reassuring to think that $\phi(R)$ could still be arbitrarily close to $1$ for such $R$.  But no, \thref{nogap} tells us that this is not the case, as shown below.

\begin{cor}
Let $\phi$ be a state on C$^*$-algebra $A$ of real rank zero such that $\mathcal{P}(A)$ has no $(\omega,\omega)$-gaps.  For any $(P_n)\subseteq\mathcal{P}(A)$, $\{\phi(Q):Q\in\mathcal{P}(A)\wedge\forall n\in\omega(Q\leq P_n)\}$ has a maximum.
\end{cor}

\begin{proof} Take $(Q_n)\subseteq\mathcal{P}(A)$ such that $Q_n\leq P_m$, for all $m,n\in\omega$, and \[\phi(Q_n)\uparrow\sup\{\phi(Q):Q\in\mathcal{P}(A)\wedge\forall n\in\omega(Q\leq P_n)\}.\]  By \thref{arbcount}, we may replace $(P_n)$ and $(Q_n)$ with equivalent (possibly non-strictly) decreasing $(P'_n)$ and increasing $(Q'_n)$.  We can then obtain $R\in\mathcal{P}(A)$ such that $Q'_n\leq R\leq P'_n$, for all $n\in\omega$, (if one of these sequences is eventually constant, this is immediate, otherwise use the fact that $\mathcal{P}(A)$ has no $(\omega,\omega)$-gaps), giving $Q_n\leq R\leq P_n$ and hence $\phi(Q_n)\leq\phi(R)$, for all $n\in\omega$. \end{proof}

\section{Upper Bounds}\label{ub}

For projections $P$ and $Q$, we have $P\leq Q\Leftrightarrow Q^\perp\leq P^\perp$, and hence each result in \S\ref{lb} for lower bounds has an equivalent formulation in terms of upper bounds.  Also, by applying (\ref{sigmaPQP}) and (\ref{sigmaPQperpP}), we get $\sigma(PQ)\cap(0,1)=\sigma(P^\perp Q^\perp)\cap(0,1)$ and hence, when $A$ has real rank zero and $\mathcal{P}(A)$ is $\sigma$-closed,
\begin{equation}\label{glblub}
P\wedge Q\in A\qquad\Leftrightarrow\qquad P\vee Q\in A,
\end{equation}
by \thref{finglbpi}.  We can also prove the following analog of \thref{cglb} for upper bounds, which actually has some added strength owing to the fact that $\mathcal{R}(Q_n)\subseteq\sum_{k\leq n}\mathcal{R}(P_k)$, for all $n\in\omega$.

\begin{thm}\thlabel{club}
Let $H$ and $H'$ be Hilbert spaces, let $A\subseteq\mathcal{B}(H)$ be a C$^*$-algebra of real rank zero containing the identity operator and let $\pi:A\rightarrow\mathcal{B}(H')$ be a homomorphism such that $\pi[A]$ has the $\omega$-property.  For any $(P_n)\subseteq\mathcal{P}(A)$, we have (possibly non-strictly) increasing $(Q_n)\subseteq\mathcal{P}(A)$ with $\mathcal{R}(Q_n)\subseteq\sum_{k\leq n}\mathcal{R}(P_k)$, for all $n\in\omega$, and $\bigvee\pi(P_n)=\bigvee\pi(Q_n)$.  In particular, if $(\pi(P_n))$ has an l.u.b. then it has one of the form $\pi(P)$, where $\mathcal{R}(P)\subseteq\sum\mathcal{R}(P_n)$
\end{thm}

\begin{proof} Use the proof of \thref{finglbpi} to prove the first part for pairs, i.e. to show that, for any $P,Q\in\mathcal{P}(A)$, we can find $(Q_n)\subseteq\mathcal{P}(A)$ such that $\mathcal{R}(Q_n)\subseteq\mathcal{R}(P)+\mathcal{R}(Q)$, for all $n\in\omega$, and also $\pi(P)\vee\pi(Q)=\bigvee\pi(Q_n)$.  The key to proving $\mathcal{R}(Q_n)\subseteq\mathcal{R}(P)+\mathcal{R}(Q)$ is to note that $\mathcal{R}(E_{Q^\perp P^\perp Q^\perp}(t))=\mathcal{R}(Q)+\mathcal{R}(PE_{PQP}(t))$, for all $t\in[0,1)$.  Proving this for finite collections is done by induction and then the countable case follows from this, as in the proof of \thref{cglb}. \end{proof}

If $P$, $Q$ and $R$ are projections then $R\leq P\vee Q$ is equivalent to $\mathcal{R}(R)\subseteq\overline{\mathcal{R}(P)+\mathcal{R}(Q)}$.  However, $\mathcal{R}(R)\subseteq\overline{\mathcal{R}(P)+\mathcal{R}(Q)}$ does not necessarily imply $\pi(R)\leq\pi(P)\vee\pi(Q)$ or even $\pi(R)\leq\pi(S)$ for a projection $S$ such that $\pi(S)\geq\pi(P),\pi(Q)$.  Indeed, for the $P$ and $Q$ in (\ref{badPQ}), $\overline{\mathcal{R}(P)+\mathcal{R}(Q)}=H$ even though $\pi(P)=\pi(Q)\neq1$, where $\pi:\mathcal{B}(H)\rightarrow\mathcal{C}(H)$ is canonical.  The interesting thing, though, is that the slightly stronger statement $\mathcal{R}(R)\subseteq\mathcal{R}(P)+\mathcal{R}(Q)$ will actually yield $\pi(R)\leq\pi(S)$ for any such $S$, as we now show.  This also yields characterizations of those $P$ such that $\mathcal{R}(P)\subseteq\sum\mathcal{R}(P_n)$ and $\pi(P)$ is an l.u.b. of $(\pi(P_n))$ (which holds for some $P$ if $(\pi(P_n))$ has an l.u.b., by \thref{club}).

\begin{thm}\thlabel{PPn}
Let $A(\subseteq\mathcal{B}(H))$ be a von Neumann algebra\footnote{Actually, the theorem holds more generally for all C$^*$-algebras of real rank zero, with a slightly different proof \textendash\, see \cite{s} Corollary 6.4}, let $\pi$ be a homomorphism from $A$ onto a C$^*$-algebra $B$ and take $(P_n)\subseteq\mathcal{P}(A)$ and $P\in\mathcal{P}(A)$ with $\mathcal{R}(P)\subseteq\sum\mathcal{R}(P_n)$.  If $Q\in\mathcal{P}(A)$ satisfies $\pi(P_n)\leq\pi(Q)$, for all $n\in\omega$, then $\pi(P)\leq\pi(Q)$.  Moreover, the following are equivalent.
\begin{enumerate}
\item\thlabel{club1} $\pi(P)$ is the l.u.b. of $(\pi(P_n))$ in $\mathcal{P}(B)$.
\item\thlabel{club2} $\pi(P)\geq\pi(Q)$ whenever $Q\in\mathcal{P}(A)$ and $\mathcal{R}(Q)\subseteq\sum\mathcal{R}(P_n)$.
\item\thlabel{club3} $\pi(P)=\pi(Q)$ whenever $Q\in\mathcal{P}(A)$ and $\mathcal{R}(P)\subseteq\mathcal{R}(Q)\subseteq\sum\mathcal{R}(P_n)$.
\end{enumerate}
\end{thm}

\begin{proof} Take positive $(r_n)$ such that $\sum r_n<\infty$.  Then $\sum r_nR_n$ will be a bounded linear operator on (the \emph{external} direct sum) $\bigoplus\mathcal{R}(P_n)$, where $R_n$ is the projection onto the $n^\mathrm{th}$ coordinate.  Thus $N=\{(v_n)\in\bigoplus\mathcal{R}(P_n):\sum r_nv_n=0\}=\mathcal{N}(\sum r_nR_n)$ will be closed.  Now let $R$ be the map taking each $v\in\sum\mathcal{R}(P_n)$ (or even $v\in\mathcal{R}(\sum r_nR_n)$) to the unique $(v_n)\in\bigoplus\mathcal{R}(P_n)$ that is orthogonal to $N$ and satisfies $\sum r_nv_n=v$.  Now note that $R$ is closed because if $v_n\rightarrow v$ and $Rv_n\rightarrow(w_m)$ then, as $Rv_n\perp N$ for all $n\in\omega$, $(w_m)\perp N$ and $v_n=\sum_mr_mR_mRv_n\rightarrow\sum_mr_mw_m=v$ so $Rv=(w_n)$.  By the closed graph theorem, $R$ is bounded on $\mathcal{R}(P)$.  Furthermore, we claim that, for all $n\in\omega$ and $S\in A'$, we have $SR_nRP=R_nRPS$ and hence $R_nRP\in A''=A$.  To see this, take $S\in A'$ and note that, as $SP=PS$, it suffices to prove that $SR_nRw=R_nRSw$, for all $n\in\omega$ and $w\in\mathcal{R}(P)$.  Take $n\in\omega$ and $w\in\mathcal{R}(P)$ and note that $w=r_0w_0+\ldots+r_mw_m$ for some $w_0,\ldots,w_m$ such that $w_k\in\mathcal{R}(P_k)$, for all $k\leq m$, and $(w_k)\perp N$ (setting $w_k=0$, for all $k>m$).  It follows immediately that $SR_nRw=Sw_n$.  If $(x_k)\in N$ then $\sum r_kx_k=0$, so $\sum r_kS^*x_k=0$ and hence $(S^*x_k)\in N$.  As $(w_k)\perp N$, $\sum\langle Sw_k,x_k\rangle=\sum\langle w_k,S^*x_k\rangle=0$, i.e. $(Sw_k)\perp N$.  Thus $R_nRSw=Sw_n$, and the claim is proved.  It follows that $\pi(Q^\perp R_nRP)=\pi(Q^\perp P_n)\pi(R_nRP)=0$, for all $n\in\omega$, and hence $\pi(Q^\perp P)=\pi(Q^\perp(\sum r_nR_nR)P)=\sum r_n\pi(Q^\perp R_nRP)=0$.

Thus the first part is proved, and to prove the equivalence of the last three statements, note first that \ref{club1}$\Rightarrow$\ref{club2} follows from the first part.  Conversely, if \ref{club2} holds then $\pi(P)\geq\pi(P_n)$, for all $n\in\omega$, which, together with the first part again, implies that $\pi(P)$ is an l.u.b. of $(\pi(P_n))$ in $\mathcal{P}(B)$.  The \ref{club2}$\Rightarrow$\ref{club3} part is immediate.  Conversely, say \ref{club2} fails, i.e. $\pi(P)\ngeq\pi(Q)$ for some $Q\in\mathcal{P}(A)$ with $\mathcal{R}(Q)\subseteq\sum\mathcal{R}(P_n)$.  Then $||PE^\perp_{QP^\perp Q}(\delta)||\leq\sqrt{1-\delta}$, where $0<\delta<||\pi(P^\perp Q)||^2$, so $\mathcal{R}(P)+\mathcal{R}(E^\perp_{QP^\perp Q}(\delta))$ is closed, and if $R$ denotes the projection onto this subspace then $R\in\mathcal{P}(A)$ and $\mathcal{R}(P)\subseteq\mathcal{R}(R)\subseteq\sum\mathcal{R}(P_n)$, even though $||\pi(P^\perp R)||\geq||\pi(P^\perp E^\perp_{QP^\perp Q}(\delta))||=||\pi(P^\perp Q)||$ and hence $\pi(P)\ngeq\pi(R)$. \end{proof}

We can now show that the canonical homomorphism to the Calkin algebra preserves l.u.b.s (of countable subsets) precisely when the span of the corresponding subspaces is closed.

\begin{cor}\thlabel{findim}
Let $H$ and $H'$ be Hilbert spaces, let $\pi:\mathcal{B}(H)\rightarrow\mathcal{B}(H')$ be a homomorphism such that $\pi^{-1}[\{0\}]=\mathcal{K}(H)$, let $(P_n)\subseteq\mathcal{P}(\mathcal{B}(H))$ and take $P\in\mathcal{P}(\mathcal{B}(H))$ with $\mathcal{R}(P)\subseteq\sum\mathcal{R}(P_n)$.  Then $\pi(P)=\bigvee\pi(P_n)$ if and only if $\mathcal{R}(P)$ has no closed infinite dimensional extension in $\sum\mathcal{R}(P_n)$.  Hence $\sum\mathcal{R}(P_n)$ is closed if and only if $\pi(\bigvee P_n)=\bigvee\pi(P_n)$.
\end{cor}

\begin{proof} The \ref{club1}$\Leftrightarrow$\ref{club3} part of \thref{PPn} proves the first part of this corollary.  It (or the \ref{club2}$\Rightarrow$\ref{club1} part) also shows that if $\sum\mathcal{R}(P_n)$ is closed then $\pi(\bigvee P_n)=\bigvee\pi(P_n)$.  Conversely, if we have $\pi(\bigvee P_n)=\bigvee\pi(P_n)$ then, for the $(Q_n)$ in \thref{club}, we must have $\pi(Q_m)=\bigvee\pi(P_n)=\pi(\bigvee P_n)$, for some $m\in\omega$, because $\pi[\mathcal{B}(H)]$ is isomorphic to $\mathcal{C}(H)$ and hence $\mathcal{P}(\pi[\mathcal{B}(H)])$ is $\sigma$-closed.  As $\mathcal{R}(Q_m)\subseteq\sum\mathcal{R}(P_n)\subseteq\overline{\sum\mathcal{R}(P_n)}=\mathcal{R}(\bigvee P_n)$, we see that $\sum\mathcal{R}(P_n)$ is a finite dimensional extension of the closed subspace $\mathcal{R}(Q_m)$ and hence itself closed. \end{proof}

\section{The Order on Self-Adjoint Operators}

While not the main focus of this paper, we prove one result about the order on all self-adjoint operators $\mathcal{S}(A)$ for an arbitrary unital C$^*$-algebra $A$, improving on \cite{j} Corollary 4.

\begin{thm}
Let $A$ be an arbitrary unital C$^*$-algebra.  If $P,Q\in\mathcal{P}(A)$ do not commute then there exists $S\in\mathcal{S}(A)$ such that $S$ and $0$ are incomparable and $(\{0,S\},\{P,Q\})$ is a gap in $\mathcal{S}(A)$.
\end{thm}

\begin{proof} As $P$ and $Q$ do not commute, $PQP$ is not a projection and we have $r\in(0,1)\cap\sigma(PQP)$.  Let $f$ be a continuous function on $\mathbb{R}$ with the following properties
\begin{enumerate}
\item $f(r)>0$,
\item $f(t)\leq-1$, for all $t\leq r/2$, and
\item $f(t)\leq r/4$, for all $t\geq r/2$.
\end{enumerate}
Let $S=\int f(t)dE_{PQP}(t)$ and take $v\in H$.  If $\langle Sv,v\rangle\leq 0$, then we certainly have $\langle Sv,v\rangle\leq\langle Qv,v\rangle$.  Otherwise, as \[\langle Sv,v\rangle=\langle SE_{PQP}(r/2)v,v\rangle+\langle SE^\perp_{PQP}(r/2)v,v\rangle\leq (r/4)||E^\perp_{PQP}(r/2)v||^2-||E_{PQP}(r/2)v||^2,\] we must have $||E_{PQP}(r/2)v||\leq(\sqrt{r}/2)||E^\perp_{PQP}(r/2)v||\leq\sqrt{r/2}||E^\perp_{PQP}(r/2)v||$.  Thus, as
\[||Qv||\geq||QE^\perp_{PQP}(r/2)v||-||E_{PQP}(r/2)v||\geq\sqrt{r/2}||E^\perp_{PQP}(r/2)v||-||E_{PQP}(r/2)v||,\] we have $\langle Qv,v\rangle=||Qv||^2\geq(\sqrt{r/2}||E^\perp_{PQP}(r/2)v||-||E_{PQP}(r/2)v||)^2$.  Therefore, to show that $S\leq Q$, it suffices to show that \[(r/4)x^2-y^2\leq(\sqrt{r/2}x-y)^2=rx^2/2-\sqrt{2r}xy+y^2,\] for all $r,x,y\in\mathbb{R}$.  But this follows immediately from the fact that \[0\leq rx^2/4-\sqrt{2r}xy+2y^2=(\sqrt{r}x/2-\sqrt{2}y)^2,\] for all $r,x,y\in\mathbb{R}$.

On the other hand, $P$ commutes with $PQP$ and hence with $S$, as $S$ is in the closure of the algebra generated by $PQP$ and $1$.  Thus, \[\langle Sv,v\rangle-\langle Pv,v\rangle=\langle SP^\perp v,P^\perp v\rangle\leq-||P^\perp v||^2,\] for all $v\in H$, the last inequality coming from $\mathcal{R}(P)^\perp\subseteq\mathcal{R}(E_{PQP}(0)))\subseteq\mathcal{R}(E_{PQP}(r/2)))$.  Thus we also have $S\leq P$.

However, for $\epsilon>0$ small enough that $\epsilon<1-r$ and $f(t)>0$, for all $t\in[r-\epsilon,r+\epsilon]$, we can find non-zero \[v\in\mathcal{R}(E_{PQP}(r+\epsilon)-E_{PQP}(r-\epsilon))\subseteq E_{PQP}(1-)=(\mathcal{R}(P)\cap\mathcal{R}(Q))^\perp.\]  We therefore have $\langle Sv,v\rangle>0$ even though $\langle(P\wedge Q)v,v\rangle=0$.  Thus $S\nleq P\wedge Q$, even though $P\wedge Q$ is a g.l.b. of $P$ and $Q$ in the set of all $T\in\mathcal{S}(\mathcal{B}(H))$ with $T\geq0$, by \cite{j} Lemma 2.  Thus the pregap $(\{0,S\},\{P,Q\})$ can not be interpolated.

The $v$ of the last paragraph witnesses the fact that $S\nleq0$, while any non-zero $v\in\mathcal{R}(P)^\perp\subseteq\mathcal{R}(E_{PQP}(0))$ (note that $P\neq 1$, as $P$ and $Q$ do not commute, and hence $P^\perp\neq 0$) will witness the fact that $S\ngeq0$, i.e. $S$ and $0$ are incomparable. \end{proof}

\begin{cor}\thlabel{PQncom}
Let $A$ be an arbitrary unital C$^*$-algebra.  If $P,Q\in\mathcal{P}(A)$ have a g.l.b. in $\mathcal{S}(A)$ then they necessarily commute.
\end{cor}

\end{document}